\documentclass[12pt]{amsart}
\usepackage{amsfonts}
\usepackage{amssymb}
\usepackage[utf8]{inputenc}
\usepackage[headings]{fullpage}

\usepackage[all]{xy}

\usepackage[cmyk]{xcolor}
\usepackage[colorlinks=true,linkcolor=blue, citecolor=blue,  urlcolor=blue]{hyperref}

\newtheorem{theorem}{Theorem}[section]
\newtheorem{lemma}[theorem]{Lemma}
\newtheorem{corollary}[theorem]{Corollary}
\newtheorem{proposition}[theorem]{Proposition}
\newtheorem{remark0}[theorem]{Remark}
\newtheorem{example0}[theorem]{Example}
\newtheorem{definition}[theorem]{Definition}

\newenvironment{example}{\begin{example0}\rm}{\end{example0}}
\newenvironment{remark}{\begin{remark0}\rm}{\end{remark0}}

\newcommand{\defref}[1]{Definition~\ref{#1}}
\newcommand{\propref}[1]{Proposition~\ref{#1}}
\newcommand{\thmref}[1]{Theorem~\ref{#1}}
\newcommand{\lemref}[1]{Lemma~\ref{#1}}
\newcommand{\corref}[1]{Corollary~\ref{#1}}
\newcommand{\exref}[1]{Example~\ref{#1}}

\newcommand{\remref}[1]{Remark~\ref{#1}}

\newcount\minleft
\newcount\timehour
\def\thetime{\timehour=\time
\divide\timehour by60 \minleft=\timehour \multiply\minleft by -60
\advance\minleft by\time \ifnum\time>720\advance\timehour
by-12\fi\relax
\number\timehour:\ifnum\minleft<10 %
    0\fi\relax\number\minleft
    \ifnum\time>720~pm \else~am\fi}
\def\res{{\mathbf{k}}}

\def\HF{{\operatorname{H\!F}}}

\def\Max{\operatorname{max}}
\def\deg{\operatorname{deg}}
\def\dim{\operatorname{dim}}
\def\ann{\operatorname{Ann}}

\def\length{\operatorname{Length}}

\def\rank{\operatorname{rank}}

\begin{document}

\title[Inverse system of Gorenstein set of points]{
Inverse system of Gorenstein   points in ${\mathbb P^n_{\res}}$}
\author[J. Elias]{J. Elias $^*$}
\thanks{${}^{*}$
Partially supported by
PID2019-104844GB-I00}
\address{Joan Elias
\newline \indent Departament de Matem\`{a}tiques i Inform\`{a}tica
\newline \indent Universitat de Barcelona (UB)
\newline \indent Gran Via 585, 08007
Barcelona, Spain}  \email{{\tt elias@ub.edu}}

\author[M. E. Rossi]{M. E. Rossi $^{**}$}
\thanks{${}^{**}$
Partially supported by PRIN-MIUR 2020355B8Y.}

\address{Maria  Evelina  Rossi
\newline \indent Dipartimento di Matematica
\newline \indent Universit\`{a} di genova
\newline \indent Via Dodecaneso 35, 16146 Genova, Italy}
\email{{\tt rossim@dima.unige.it}}

%%\date{\today}
\date{\today }

\begin{abstract}   Given a set of distinct points  $X=\{P_1, \dots,P_r\} $ in   $ \mathbb P^n_{\res}, $ in this paper we characterize being $X$   arithmetically Gorenstein through the ``special" structure of the inverse system of the defining ideal $I(X) \subseteq R= \res[X_0, \dots, X_n].$ We describe the corresponding Zariski locally closed subset   of $\mathbb (\mathbb P_{\res}^{n})^r$ parametrized by the coordinates of the points. Several examples are given to show the effectiveness of the results.  As a consequence of the main result we characterize the Artinian Gorenstein rings which are the Artinian reduction  of Gorenstein points.  The problem is of  interest  in several areas of research,  among others  the $G$-linkage, the Waring decomposition, the smoothability of  Artinian algebras and the identifiability of symmetric tensors.

\end{abstract}

\maketitle

%\tableofcontents

%%%%%%%%%%%%%%%%%%%%%%%%%%%%%%%%%%%%%%%%%%%%%%%%%%%%%%%%%%%%%%%%%%%%%%%%%%%%%%%%%%%%%%%%%%%%%
%%%%%%%%%%%%%%%%%%%%%%%%%%%%%%%%%%%%%%%%%%%%%%%%%%%%%%%%%%%%%%%%%%%%%%%%%%%%%%%%%%%%%%%%%%%%%
%%%%%%%%%%%%%%%%%%%%%%%%%%%%%%%%%%%%%%%%%%%%%%%%%%%%%%%%%%%%%%%%%%%%%%%%%%%%%%%%%%%%%%%%%%%%%
%%%%%%%%%%%%%%%%%%%%%%%%%%%%%%%%%%%%%%%%%%%%%%%%%%%%%%%%%%%%%%%%%%%%%%%%%%%%%%%%%%%%%%%%%%%%%
%%%%%%%%%%%%%%%%%%%%%%%%%%%%%%%%%%%%%%%%%%%%%%%%%%%%%%%%%%%%%%%%%%%%%%%%%%%%%%%%%%%%%%%%%%%%%
\section{Introduction}

 Let  $X=\{P_1, \dots,P_r\} $ be a set of non degenerate distinct points in  $ \mathbb P^n_{\res}  $    where $\res$ is an algebraically closed field of characteristic zero.
The defining ideal $I(X) \subseteq  R= \res{}[x_0, \dots, x_n]$  is  saturated and  $R/I(X) $ is a  Cohen-Macaulay  one-dimensional  standard graded $\res$-algebra.
Little is known   how to obtain a reduced zero-dimensional scheme $X  $ which is arithmetically  Gorenstein, especially if $X$ is not a complete intersection.
Many authors give geometric constructions of some particular families of arithmetically Gorenstein schemes (e.g. \cite{BCC22}, \cite{BDNS04}, \cite{GM97}, \cite{MN03},  \cite{MP97}, \cite{CI12}), but the problem remains open in general.
Notice that  a   set of $r$ {\it{generic }} points  is never Gorenstein, provided  $r\ne n+2$.
Recall that was noted by Eisenbud and Popescu (\cite{EP00}, Theorem 7.3) that a set $X$ of $2n+2$  points in $\mathbb P^n_{\res}  $  is arithmetically Gorenstein if and only if they are self-associated. This is the case of an hyperplane sections of canonically embedded curves.
In general, a  geometric
or algebraic description of arithmetically Gorenstein reduced schemes  is not understood.
For instance this lack is one of the main obstacles in the Gorenstein liaison (or $G$-linkage)  of zero-dimensional schemes.
Recall that two   schemes X and Y in $\mathbb P^n_{\res}$ that are reduced and without common components are directly Gorenstein-linked if their union is arithmetically Gorenstein  (e.g.  \cite{EHS15}, \cite{Har00}, \cite{KMMNP01}).

We can say that our main purpose would be to ``construct" Gorenstein points through the structure of Macaulay's inverse system, rather than to prove that a given set of points is arithmetically Gorenstein.

As a continuation of \cite{ER17}, \cite{ER21}, taking advantage of the knowledge on the inverse system of one-dimensional Gorenstein $\res$-algebras, in this paper  we describe Macaulay's inverse system of Gorenstein points in terms of their coordinates,  see \thmref{char-points-gor}.
We will prove, by analogy with the Artinian case, that the inverse system of Gorenstein points can be characterized in terms of a suitable homogeneous polynomial,  despite the fact that  $I(X)^{\perp}$ is a not finitely generated $R$-module.
It is clear that for being Gorenstein the points must be somehow ``special" and we describe the corresponding Zariski locally closed subset (eventually empty) of $(\mathbb P_{\res}^{n})^r$  parametrizing  the coordinates of the points, see Proposition \ref{steGor} and Remark \ref{zariski}.

Beside the $G$-linkage, the problem we face in this paper is also related to the  study of the Waring problem  and to  the identifiability of symmetric tensors which plays a crucial role in Algebraic Statistic (see \cite{COV17}).
This   because the Artinian reduction of $R/I(X) $ where $X$ is arithmetically Gorenstein, is Gorenstein as well.
The inverse system of any Artinian reduction of  $I(X) $ becomes very concrete.
By Matlis duality the dual module is generated by a homogeneous form $F$  and  $F$ has a  representation as a sum of powers of linear forms corresponding  to the points, see  Proposition \ref{reductionPoints}.
This is the connection with the celebrated  Big Waring Problem and though we will not discuss here about this topic, see    for instance   \cite{Ger96} and \cite{IK99}. Inverse systems of Artinian graded rings coming from points are put to great use also in the theory of splines approximation (e.g. \cite{GS98}) and in the study of Weak Lefschetz Property (e.g. \cite{HSS11}, \cite{MMRN10}).

\vskip 2mm
Let us describe briefly the organization of the paper.
In Section $2$ we review the main results contained  in the   papers \cite{ER17} and \cite{ER21} which are the basic tools in this paper.

In Section $3$ we present one of the main results  of the paper. In particular in Proposition   \ref{G-points}  we specialize  the general notion of $G$-admissibility (see Section $2$)      to the defining ring of a set $X$ of points in the projective space.
Finally in \thmref{char-points-gor} and \remref{Pi} we present effective   conditions  for being $X$ arithmetically Gorenstein  and we describe explicitly  its inverse system.
This result can be
 considered as a strengthening of \cite[Theorem 2.2]{Toh14}.

Taking advantage of Theorem \ref{char-points-gor}, in Section $4$ we describe in an explicit way the Zarisky locally closed locus parameterizing  the   points
$X=\{P_1,\dots, P_r\}\subset  \mathbb P^n_{\res}$  which are arithmetically Gorenstein, see Remark \ref{zariski} and  \propref{steGor}.

As a consequence of the previous results, in Section $5$ we characterize the Artinian Gorenstein rings which can be lifted to Gorenstein points.   Any homogeneous form $F$ of degree $s$ determines via the Macaulay inverse system an Artinian Gorenstein ring, but it is not  true that it is necessarily the Artinian reduction of a set of Gorenstein points.
In other words we ask under which conditions an Artinian ring is the Artinian reduction of a Gorenstein reduced set $X$ of points in $\mathbb P^n_{\res}$.
We   answer this question in  \propref{reductionPoints}.
This was also the main issue in \cite{Toh14}.
From the knowledge on the Waring problem (see \cite{AH00}) we can give a family of Artinian Gorenstein rings which cannot be lifted to Gorenstein points, \propref{NoLifGor}. This problem is also related to the smoothability of Artinian algebras, see for instance \cite{IK99}, \cite{Jel14}, \cite{JKK19} and more recently \cite{Sza22}.

\medskip
Through the paper several examples are given. The computations are performed  by using the computer algebra system  Singular  \cite{DGPS}  and the Singular library {\sc{inverse-syst.lib}}, \cite{E-InvSyst14}.

\medskip
\noindent
{\sc Acknowledgements.}
This  paper was written during the visit of the first author to the Department of Mathematics of Genova University.
He would like to thank  the Department  for its support and hospitality.
The authors thank A. De Stefani, J. Jelisiejew  and R.M. Miró-Roig for their useful comments.
The authors thanks the anonymous referee because his suggestions significantly  improved the presentation of the paper.

%%%%%%%%%%%%%%%%%%%%%%%%%%%%%%%%%%%%%%%%%%%%%%%%%%%%%%%%%%%%%%%%%%%%%%%%%
%%%%%%%%%%%%%%%%%%%%%%%%%%%%%%%%%%%%%%%%%%%%%%%%%%%%%%%%%%%%%%%%%%%%%%%%%
%%%%%%%%%%%%%%%%%%%%%%%%%%%%%%%%%%%%%%%%%%%%%%%%%%%%%%%%%%%%%%%%%%%%%%%%%
%%%%%%%%%%%%%%%%%%%%%%%%%%%%%%%%%%%%%%%%%%%%%%%%%%%%%%%%%%%%%%%%%%%%%%%%%
\bigskip
\section{Preliminaries}
 Let $V$ be a $\res$-vector space of dimension $n$
with basis    $\langle x_1, \dots, x_n \rangle,$ then we denote  by   $V^* = \langle y_1, \dots, y_n \rangle$   the  dual basis.
Let $R= Sym_{\cdot}^{\res} V = \oplus_{i \ge 0} Sym_i^{\res }V $ be  the  standard graded polynomial ring in $n$ variables over $\res$ and $\Gamma = D_{\bf \cdot}^{\res}(V^*) = \oplus_{i\ge 0} D_i^{\res}(V^*) = \oplus_{i\ge 0}  {\rm Hom}_{\res} (R_i, \res) $    be the graded $R$-module of graded $\res$-linear homomorphisms from $R$ to $\res, $ hence
   $\Gamma \simeq  \res_{DP} [y_1, \dots, y_n] $ the divided power ring.

  We recall   that $\Gamma $ is an  $R$-module with $R$ acting   on  $\Gamma$ by {\it {derivation}}  denoted   by $\circ$. We will think  the elements of $R$  as representing partial differential operators and
the polynomials of $\Gamma$  as the “real” polynomials on which the differential operators act.
This action is sometimes called the “apolarity” action of $R$ on $\Gamma$. We begin with a precise
definition of this action by
$$
\begin{array}{ cccc}
\circ: & R  \times \Gamma  &\longrightarrow &  \Gamma   \\
                       &       (f , g) & \to  &  f  \circ g = f (\partial_{y_1}, \dots, \partial_{y_n})(g)
\end{array}
$$

\noindent
where $  \partial_{y_i} $ denotes the partial derivative with respect to $y_i.$ This action  can be defined    bilinearly  from that for monomials.
If  we  denote by $ x^{\alpha}= x_1^{\alpha_1} \cdots x_n^{\alpha_n} $ then
$$
x^{\alpha} \circ y^{\beta} =
\left\{
\begin{array}{ll}
\frac{\beta !}{(\beta-\alpha)!}\; y^{ \beta-  \alpha}   & \text{ if }  \beta_i \ge \alpha_i \text{ for } i=1,\dots, n\\ \\
0 & \text{ otherwise}
\end{array}
\right.
$$
$\frac{\beta !}{(\beta-\alpha)!}=\prod_{i=1}^n  \frac{\beta_i !}{(\beta_i-\alpha_i)!}$.
Notice that the apolarity action induces a non-singular $\res$-bilinear pairing:
$$
\begin{array}{ cccc}
\circ: & R_j  \times \Gamma_j   &\longrightarrow &  \res
\end{array}$$
for every $j \ge 0.$

Recall that the injective hull $E_R(\res) $ of $\res$ as an $R$-module is isomorphic as an $R$-module to the divided power ring $\Gamma$ (see \cite{Gab59}, \cite{Nor72b}).
For detailed information   see also \cite{EisGTM}, \cite{Ems78}, \cite[Appendix A]{IK99}.
If $I\subset R$ is an ideal of $R, $ then  $(R/I)^\vee={\rm Hom}_R(R/I,E_R(\res))$ is the $R$-submodule of $\Gamma$
$$
{I^{\perp}} =\{ F \in \Gamma \ |\  I \circ F = 0 \  \}.
$$
This submodule of $\Gamma $  is called {\it{Macaulay's inverse system of $I$.}}
Given an $R$-submodule $W$ of $\Gamma, $ then the dual $W^\vee={\rm Hom}_R(W,E_R(\res))$ is the ring $R/\ann_R(W)$ where
$$ \ann_R(W)= \{ g \in R \ \mid \ g \circ F= 0 \ \mbox{ for \ all \ } F \in W\}
$$
 is an  ideal of $R$.

If $I$ is a homogeneous ideal of $R$ (resp. $W$ is generated by homogeneous polynomials of $\Gamma$),  then $I^{\perp}$ is generated by homogeneous polynomials of $\Gamma$ (resp. $\ann_R(W) $ is an homogeneous ideal of $R$)  and
$ I^{\perp} = \oplus I_j ^{\perp} $ where $I_j ^{\perp} = \{ F \in \Gamma_j \mid  \ g \circ F=0 \ \ \mbox{ for \ all \ }  g \in I_j \}$, i.e. $(I^{\perp})_j=I_j ^{\perp}$, see \cite{Ger96}, Proposition 2.5.
Notice that, in particular, the Hilbert function of $R/I$ can be computed by $I^{\perp}, $ see for instance \cite[Section 2]{ER17}.

 Macaulay in \cite[IV]{Mac16} proved a particular case of Matlis duality, called Macaulay's correspondence,
between the
  ideals $I\subseteq R  $ such that $R/I$ is an Artinian local ring and $R$-submodules  $W=I^{\perp}$ of $\Gamma $ which are finitely generated.
Macaulay's  correspondence is an effective method for computing Gorenstein Artinian rings, see  \cite{CI12}, Section 1, \cite{Iar94}, \cite{Ger96} and \cite{IK99}. Macaulay proved that   Artinian  Gorenstein $\res$-algebras $A=R/I$ of socle degree $s$ correspond to cyclic $R$-submodules  of $\Gamma$ generated by a polynomial $F\neq 0$ of degree $s$.
  See \cite{Ems78}, \cite{EI78}, \cite{IK99} for more details on inverse systems.

\medskip
 The authors extended  Macaulay's correspondence characterizing the $d$-dimensional local Gorenstein $\res$-algebras in terms of suitable submodules of $\Gamma$, see \cite{ER17}. This result in the one-dimensional case will be our starting point and for completeness we include here the statement.

In the following given a family $\mathcal F=\{F_j ; j\in J\}$ of elements of $\Gamma$ we denote by
$\langle F_j ; j\in J \rangle$ (resp.  $\langle F_j ; j\in J \rangle_{\res}$) the $R$-submodule (resp. $\res$-vector subspace) of $\Gamma$ generated by $\mathcal F.$

\medskip
\begin{definition}
\label{G-modules}  An   $R$-submodule $M $ of  $ \Gamma  $    is called $G$-admissible if it admits a countable system of  generators $\{F_l\}_{ l\in \mathbb N_+} $ satisfying  the following conditions
\begin{enumerate}
\item[(1)] There exists a linear form   $z \in R$ such that for all $l\in \mathbb N_+$
$$
z \circ F_l =
\left\{
\begin{array}{ll}
 F_{l-1}   & \text{ if }  l> 1 \ \\
0 & \text{ otherwise.}
\end{array}
\right.
$$
\item[(2)] $\ann_R(\langle F_l\rangle)\circ F_{l+1}=\langle F_{1}\rangle$ for all $l\in \mathbb N_+$.
\end{enumerate}

\noindent
If this is the case, we say that $M = \langle F_l, l\in \mathbb N_+\rangle$ is a
$G$-admissible $R$-submodule of $\Gamma$ with respect to the linear form   $z \in R$.
\end{definition}

%%%%%%%%%%%%%%%%%%%%%%%%%%%%%%%%%%%%%%%%%%%%%%%%%%%%%%%%%%%%%%%%%%%%%%%%%%%%%%%%%%%%%%%%%%%%%%%%%%%%%%%%%%%%%%%%%%%%%%%%

We present the main result of \cite{ER17} in the one-dimensional case of interest in this paper.

\medskip
\begin{theorem}
\label{bijecGor}
There is a one-to-one correspondence $\mathcal C$ between the following sets:

\noindent
(i)
 one-dimensional  Gorenstein $\res$-algebras $A=R/I$,

\noindent
(ii) non-zero $G$-admissible $R$-submodules $M=\langle F_l, l\in \mathbb N_+\rangle$ of $\Gamma$.

\noindent
In particular, given an ideal $I\subset R$ with $A= R/I$ satisfying $(i)$ and $z$ a linear regular element  modulo $I, $   then
 $$\mathcal C(A)= I^{\perp} = \langle F_l,  l\in \mathbb N_+\rangle \subset \Gamma   \ \ {\text{with}}\ \
  \langle F_l\rangle=(I+(z^l))^\perp $$ is $G$-admissible.
 Conversely, given an  $R$-submodule $M$ of $\Gamma $ satisfying  (ii), then $$\mathcal C^{-1}(M)=R/I  \ \ {\text{with}}\ \
 I= \ann_R(M) = \bigcap_{l\in \mathbb N_+}\ann_R(\langle F_l\rangle).$$
\end{theorem}

In the next result we prove that the condition $(1)$ of the Definition \ref{G-modules}  suffices for a system of generators of a Gorenstein ring to be $G$-admissible.
This result holds for arbitrary dimension, we give here the proof in the one-dimensional case.

\medskip
\begin{proposition}
\label{1-suffices}
Let $A= R/I$ be a one-dimensional Gorenstein ring with inverse system
$I^\perp$  containing a set $\mathcal F=\{F_l\}_{ l\in \mathbb N_+} $.
Let $z$ be a non-zero divisor of $A$.
If $\mathcal F$ satisfies the following conditions:
\begin{itemize}
    \item [(1)] for all $l\in \mathbb N_+$
$$
z \circ F_l =
\left\{
\begin{array}{ll}
 F_{l-1}   & \text{ if }  l> 1 \ \\
0 & \text{ otherwise,}
\end{array}
\right.
$$
\item [(2)] $\deg(F_{1})$ agrees with the socle degree of $A. $
\end{itemize}
then there exist non-zero elements
$\alpha_l\in \res^*$, $ l\in \mathbb N_+$, such that $\mathcal F'=\{\frac{1}{\alpha_l} F_l;  l\in \mathbb N_+\}$ is $G$-admissible with respect to $z$; in particular $\mathcal F$ ( and $\mathcal F'$) is a system of generators of $I^{\perp}$.
\end{proposition}
\begin{proof}
Let $\{H_l\}_{ l\in \mathbb N_+} $ be a $G$-admissible system of generators of $\mathcal C(A)$ with respect to $z$ in Theorem \ref{bijecGor} above.
Since $\deg(H_{1})$ agrees with the socle degree of $A$, from $(2)$ we get $\deg(F_{1})=\deg(H_{1})$.
Furthermore, again by  $(1)$ we deduce that
$\deg(F_l)=\deg(H_l)$ for all $l\in \mathbb N_+$.
On the other hand, since $\ann(H_l)=I+(z^l)\subset \ann{F_l}$ we deduce $F_l\in \langle H_l \rangle$.
From $\deg(F_l)=\deg(H_l)$ we deduce that
there is $\alpha_l \in \res$ such that $ F_l=\alpha_l H_l$.
If $\alpha_l=0$ then by $(1)$ we get $F_1=0$, this is not possible because $\deg(F_1)$ is the socle degree of $A$. Hence $H_l=\frac{1}{\alpha_l} F_l$ and  from this we get the claim.
\end{proof}

%%%%%%%%%%%%%%%%%%%%%%%%%%%%%%%%%%%%%%%%%%%%%%%%%%%%%%%%%%%%%%%%%%%%%%%%%%%%%%%%%%%%%%%%%%%%%%%%
%%%%%%%%%%%%%%%%%%%%%%%%%%%%%%%%%%%%%%%%%%%%%%%%%%%%%%%%%%%%%%%%%%%%%%%%%%%%%%%%%%%%%%%%%%%%%%%%%%
%%%%%%%%%%%%%%%%%%%%%%%%%%%%%%%%%%%%%%%%%%%%%%%%%%%%%%%%%%%%%%%%%%%%%%%%%%%%%%%%%%%%%%%%%%%%%%%%%
\bigskip
\section{The inverse system of Gorenstein points}

 Given a set of distinct points  $X=\{P_1,\dots, P_r\}$ in $ \mathbb P^n_{\res} $ not contained in a hyperplane,  consider  the corresponding defining ideal $I(X)\subseteq R= \res [x_0, \dots, x_n]. $
Recall that $R/I(X)$ is a Cohen-Macaulay one-dimensional graded ring and $I(X)^{\perp} $  as   graded $R$-submodule of $\Gamma=\res [y_0, \dots, y_n]$   is  not finitely generated. We have  $(I(X)^{\perp})_j= (I(X)_j)^{\perp} $ for every $j \ge 0 $ and   $$\HF_{X}(j)= \dim_{\res}(R_j/I(X)_j)= \dim_{\res} (I(X)^{\perp})_j,  $$ see \cite[Proposition 2.5]{Ger96} where we denote by $\HF_X$ the Hilbert Function of $R/I(X).$  Let $z$   be a linear form in $R$ such that $z(P_i) \neq 0$ for every $i=1,\dots, r,$ that is $z$ is a not zero divisor in $R/I(X).$ Then we will say that $R/I+(z) $ is {\it{an Artinian reduction of $X  $.}}  With this,  $h_t= \Delta \HF_X(t)= \HF_X(t) -\HF_X(t-1) $ is the Hilbert Function in degree $t$ of any Artinian reduction of $X. $   In particular $h_t =0 $ for every $t> s $ where $s $ is called {\it{the socle degree}} of $R/I(X). $ We recall that $s$ equals  also the Castelnuovo-Mumford regularity of $R/I(X).$   The vector $(h_0, \dots, h_s) $ is {\it{the $h$-vector of $R/I(X) $}} or, for short, {\it{the $h$-vector of $X.$}} Recall that $\sum_{i=0}^s h_i=r $ the number of points or equivalently the degree of $R/I(X).$

Recall that $R/I(X)$ is Gorenstein iff it is Cohen-Macaulay and the last (free) nonzero module in its minimal graded resolution is of rank $1$.
 It is known that if $X $ is arithmetically Gorenstein (that is $R/I(X)$ is a Gorenstein ring), then   the $h$-vector of $R/I(X) $ is symmetric, but the symmetry of the $h$-vector is only a necessary condition for being Gorenstein.

\medskip
\begin{example}
\label{DGO-example}
Let us consider   $X=\{P_1, \dots, P_5\} $ where   $P_1=[1,0,0,1]$, $P_2=[0,1,0,1]$, $P_3=[0,0,1,1]$, $P_4=[0,0,0,1]$, $P_5=[-1,0,0,1]$ in $ \mathbb P^3_{\res}$.
The points do not lie on $x_3=0$ and by an easy computation we see that the socle degree of $R/I(X)$ is two.  Then the scheme $X$ is not arithmetically Gorenstein, even if the  h-vector $(1,3,1)$ is symmetric.
 \end{example}

We recall here an  important result due to  Davis, Geramita and Orecchia (\cite{DGO85})   characterizing  Gorenstein points in terms of the Hilbert function, via a Cayley-Bacharach condition.
\vskip 3mm

\begin{theorem} \label{DGO}
Let $X=\{P_1, \dots,P_r\}$ be a set of $r$ distinct points   of $\mathbb P^n_{\res}$  and let $s$ be the socle degree of $R/I(X)$.
Then $X$ is arithmetically Gorenstein if and only if:
\begin{enumerate}
    \item[(1)]
    $\Delta \HF_X (t)=\Delta \HF_X(s-t)    $
    for all $t\ge 0$,
    \item[(2)] $\HF_X(t)=\HF_Y(t)$ for all $t=0,\dots,s-1$ and for all subset $Y$ of $r-1$ points of $X$.
\end{enumerate}
\end{theorem}

In this paper we are mainly interested in the construction of Gorenstein points, rather than  verifying that a given a set of points is  Gorenstein. We hope to achieve this aim through an effective  characterization of the inverse system.

Given a point $P=[a_0,\dots , a_n]\in \mathbb P^n_{\res}$ we define the dual linear form $L:=a_0 y_0+\cdots +a_n y_n$ in $\Gamma $ and we say that $L $ {\em{is the linear form associated to}} $P.$

\medskip
We present   a well known  result  concerning the inverse system of the ideal of any  set of points, see  \cite[Theorems IIA and IIB]{EI95},   \cite[Theorem 3.2]{Ger96}.

\medskip
\begin{proposition}
\label{fat-Geramita}
Let $X=\{P_1,\dots, P_r\}$ be a set of distinct points of $\mathbb P^n_{\res}$ and let $L_1, \dots, L_r$ be the corresponding associated linear forms.
Then for all $j\ge 0$
$$
(I(X)^\perp)_j=
\langle L_1^{j},\dots ,L_r^{j}\rangle_{\res} .
$$
\end{proposition}

\medskip
   It is clear that for being Gorenstein the number of points has to be compatible with a symmetric $h$-vector. But this is only a necessary condition.  If we consider $n+2$ generic points in $\mathbb P^n_{\res}, $ ($h$-vector $(1,  n,  1) $), then the defining ideal is always Gorenstein, but this is the only case where generic points work. Next case of interest is  a set $X$ of   $2n+2$   points in $\mathbb P^n_{\res} $  with $h$-vector $(1,  n, n,   1) $.  To construct  $2n+2$ Gorenstein points, a first naive idea could be to add a suitable point to a set $Y$ of $2n+1$ generic points.  Unfortunately this idea does not work if $n>3.   $    Let  $P \in  \mathbb P^n_{\res} $ a point and let $X=Y\cup\{P\}  $  where $Y$ consistes of $2n+1$ generic points. It is known that  $\HF_Y(2)=2n+1 $  and  $I(Y) $ is generated by quadrics, as we explain later.  Now since  $I(X) \subsetneq I(Y), $ then necessarily  $\HF_X(2)=2n+2 > \HF_Y(2)=2n+1 $ and hence the  $h$-vector of $X$ is   $(1,  n, n+ 1),  $ not symmetric. In particular $X$ is not aritmetically Gorenstein.
  We explain why $I(Y) $ is generated by quadrics.  This   was expected  by the ``Ideal Generation  Conjecture"  stated by Geramita and Orecchia  \cite{GO82}  and   Lorenzini \cite{Lor93}.
  They proved that there exists a Zariski open set where the conjecture holds true.
  Hence for proving the conjecture  it is enough to prove that this open set is not empty.
  This conjecture is still open in general, but it holds true     for $2n+1$ generic points. This follows by Corollary 5.1. in \cite{GL93}  for all integers $n\ge 7.$ For $n=4,5,6$  it can be verified by a straightforward computation, for instance  by  \cite{Singular}.

\medskip
 Recall that it was noted by Eisenbud and Popescu (\cite{EP00}, Theorem 7.3) that a set $X$ of $2n+2$   points in $\mathbb P^n_{\res}  $  being arithmetically Gorenstein is equivalent to be self-associated, that is any quadric containing all but one of the points of $X$ must also contain the remaining point. Self-associated sets of points arise naturally as hyperplane sections of canonically embedded curves (\cite{DO88}  Ex. 5, p. 47). More
precisely, let $C \subseteq  \mathbb P^{g-1}_{\res}  $  be the canonical embedding of a non-bielliptic curve $C$  of
genus $g$, and let $X=C \cap H$ be any transverse hyperplane section. By the Riemann-
Roch theorem, $X$ fails by 1 to impose independent conditions on quadrics, while any
proper subset of $X$ imposes independent conditions. Thus, $X$  is a self-associated  set of points
in $\mathbb P^{g-1}_{\res}.$
See also \cite{CRV01} Theorem 5.3 for a discussion on the topic.

\medskip
%%%%%%%%%%%%%%%%%%%%%%%%%%%%%%%%%%%%%%%%%%%%%%%%%%%%%%%%%%%%%%%%%%%%%%%%%%%%%%%%%%%%
%%%%%%%%%%%%%%%%%%%%%%%%%%%%%%%%%%%%%%%%%%%%%%%%%%%%%%%%%%%%%%%%%%%%%%%%%%%%%%%%%%%%
%%%%%%%%%%%%%%%%%%%%%%%%%%%%%%%%%%%%%%%%%%%%%%%%%%%%%%%%%%%%%%%%%%%%%%%%%%%%%%%%%%%%
%%%%%%%%%%%%%%%%%%%%%%%%%%%%%%%%%%%%%%%%%%%%%%%%%%%%%%%%%%%%%%%%%%%%%%%%%%%%%%%%%%%%
%%%%%%%%%%%%%%%%%%%%%%%%%%%%%%%%%%%%%%%%%%%%%%%%%%%%%%%%%%%%%%%%%%%%%%%%%%%%%%%%%%%%
%%%%%%%%%%%%%%%%%%%%%%%%%%%%%%%%%%%%%%%%%%%%%%%%%%%%%%%%%%%%%%%%%%%%%%%%%%%%%%%%%%%%
%%%%%%%%%%%%%%%%%%%%%%%%%%%%%%%%%%%%%%%%%%%%%%%%%%%%%%%%%%%%%%%%%%%%%%%%%%%%%%%%%%%%
Combining   Proposition \ref{fat-Geramita}  and Theorem \ref{bijecGor},    we aim to give an effective description of the $G$-admissible system of generators of $I(X)^{\perp}$. We start with some technical lemmas.

\medskip
\begin{lemma} \label{Lemma1}
Let $X=\{P_1,\dots, P_r\}$ be a set of distinct points of $\mathbb P^n_{\res}$ of socle degree $s$,  and let $L_1, \dots, L_r$ be the corresponding  associated linear forms.
For all $i=1,\dots,r$, there exists
an homogeneous form  $\sigma$ of degree $s$ such that
\begin{enumerate}
    \item $\sigma \circ L_i^t\neq 0$ for all $t\ge s$, and
    \item $\sigma \circ L_j^t=0$  for all $t\ge 1$ and $j\neq i$.
\end{enumerate}
 \end{lemma}
\begin{proof}
We may assume that $i=1$.
We write $Y=\{P_2,\dots, P_r\}$.
Since $Y \subsetneq X$ we have that
$$
\HF_Y(s)< \HF_X(s)=r.
$$
This implies that there exist
$\sigma \in I(Y)_s\setminus I(X)_s$, so
$\sigma(P_1)\neq 0$ and $\sigma(P_j)=0$ for all $j\ge 2$.

From $\sigma \in I(P_j)$ for $j\ge 2$ we deduce that
$\sigma \circ L_j^t=0$  for all $t\ge 1$ and $j\ge 2$ (by \propref{fat-Geramita}).

Since $\sigma \notin I(P_1)$ there exist an integer $t_0$ such that $\sigma \circ L_1^{t_0}\neq 0$.
Assume  $\sigma\circ L_1^{t_0+m}=0$ for some integer $m\ge 0$ and
 let $l \in \{0,...,n\}$ be   such that $x_l\circ L_1=c\neq 0$ (this means that the $l$-th homogeneous coordinate of the point $P_1$ is nonzero), then
$$
[c^m(t_0+m)\cdots (t_0+1)] (\sigma \circ L_1^{t_0})=\sigma\circ (x_l^m\circ L_1^{t_0+m})=
x_l^m\circ(\sigma \circ L_1^{t_0+m})=
x_l^m\circ 0=0,
$$
so $\sigma\circ L_1^{t_0}=0$; contradiction.
Hence $\sigma \circ L_1^t \neq 0$ for all $t\ge t_0$.
Take $t_0$ to be the minimum with the above property. Obviously $t_0\ge s, $ since otherwise $\sigma\circ L_1^{t_0}=0, $ as deg$(\sigma)=s.$
Assume   $t_0> s$. Then,   $\sigma\circ L_1^{t_0} $ is a non-constant homogeneous polynomial in $\Gamma$ and so, at least one of its first partial derivatives is nonzero. So there exists $l \in \{0,...,n\}$ such that $x_l\circ (\sigma \circ L_1^{t_0}) \neq 0. $ Then
$$
0\neq x_l\circ (\sigma \circ L_1^{t_0})=
\sigma\circ(x_l\circ L_1^{t_0})
=\sigma\circ(t_0 a_l L_1^{t_0-1})=
t_0 a_l (\sigma\circ L_1^{t_0-1})
$$
where $L_1=a_0 y_0+\dots+{a_n}y_n$.
Hence $\sigma \circ L_1^{t_0-1}\neq 0, $ a contradiction with the minimality of $t_0. $ So $t_0=s$.
\end{proof}

\medskip
\begin{lemma} \label{Lemma2}
Let $X=\{P_1,\dots, P_r\}$ be a set of distinct points of $\mathbb P^n_{\res}$ of socle degree $s$ and let $L_1, \dots, L_r$ be the associated linear forms.
Given an integer $u\ge 1$ and $\alpha_1,\dots ,\alpha_r\in \res^*$ let us consider a family of polynomials
$F_j=\frac{1}{(j+u)!}\sum_{i=1}^r \alpha_i L_i^{j+u}$, $j\ge 1$.
For all  integers $w\ge 1$, $i=1,\dots, r$
and $t\ge \Max\{s, s+w-u\}$ there exist a   homogeneous form $\beta_i \in R$ of degree $t+u-w$ such that
$$
L_i^w=\beta_i \circ F_t
$$
\end{lemma}
\begin{proof} Fix an integer $i \in \{1, \dots, r\} $ and consider a homogeneous form $\sigma \in R $ of degree $s$ under the hypothesis of \lemref{Lemma1}.
If $t\ge s,$  then
$$
\sigma \circ F_t=\sigma\circ\left(\frac{1}{(t+u)!}\sum_{i=1}^r \alpha_i L_i^{t+u}\right)=
\frac{1}{(t+u)!}\alpha_i (\sigma \circ L_i^{t+u}).
$$
Since $\sigma \circ L_i^{t+u}$ is a degree $t+u-s\ge 1$ non-zero homogeneous  element of   $I(P_i)^{\perp}$ there exist $c_i\in\res^*$
such that
$$
\sigma \circ L_i^{t+u}=c_i L_i^{t+u-s}.
$$
By hypothesis $\alpha_i\neq 0, $ so
$$
L_i^{t+u-s}= \frac{(t+u)!}{\alpha_i c_i} \sigma \circ F_t.
$$

Given an integer  $w\ge 1$ we consider an integer $t$ such that $t\ge s$  and $t+u-s\ge w$.
Let $l$ be an integer such that $x_l\circ L_i\neq 0$ (the point $P_i$ has $l$-th homogeneous coordinate $\ne 0$). So there is $c\in \res^*$ such that
$$
L_i^w= c x_l^{t+u-s-w}\circ L_i^{t+u-s}.
$$

From the last identity we get
$$
L_i^w= c x_l^{t+u-s-w}\circ L_i^{t+u-s}= \left( \frac{c (t+u)!}{\alpha_i c_i} x_l^{t+u-s-w} \sigma\right) \circ F_t,
$$
and we take
$$
\beta_i= \frac{c(t+u)!}{\alpha_i c_i} x_l^{t+u-s-w} \sigma.
$$
\end{proof}

\medskip
In the   following proposition we give  a different system of generators of $I(X)^{\perp}$ with respect to Proposition \ref{fat-Geramita}, suitable for checking  whether  $I(X)^{\perp}$ is  $G$-admissible.

\medskip
\begin{proposition}
\label{latte}
Let $X=\{P_1,\dots, P_r\}$ be a set of distinct points of $\mathbb P^n_{\res}$ and let $L_1, \dots, L_r$ be the associated linear forms.
Given $u\ge 1$ and $\alpha_1,\dots ,\alpha_r\in \res^*$ consider the   $R$-submodule $M$ of $\Gamma$  generated by the elements
$$   F_t:=\frac{1}{(t+u)!}\sum_{i=1}^r \alpha_i L_i^{t+u}     $$
for $t \ge 1.$
Then
\begin{enumerate}
    \item[(1)] $M= \langle L_1^t, \dots, L_r^t; \ t\ge 1 \rangle $
    \item[(2)] $ I(X)=\ann{M}$.
    \end{enumerate}
\end{proposition}
\begin{proof}
From \propref{fat-Geramita} we have
that
$$
M=\langle F_t; t\ge 1\rangle
\subset
\langle L_i^w; w\ge 1, i=1,\dots, r\rangle=
I(X)^{\perp}.
$$
On the other hand, from \lemref{Lemma2} we get that
$$
I(X)^{\perp}=
\langle L_i^w; w\ge 1, i=1,\dots, r\rangle
\subset \langle F_t; t\ge 1\rangle=M,
$$
so
$I(X)^{\perp}=M$ and then $I(X)=\ann{M}$. Then (1) and (2) follow.
\end{proof}

\medskip

\medskip
\begin{example}
\label{recover-L}
Let $X$ be as in \exref{DGO-example}.
We consider the linear forms associated to the points of $X$:
$L_1=y_0+y_3$, $L_2=y_1+y_3$, $L_3=y_2+y_3$, $L_4=y_3$ and $L_5=-y_0+y_3$.
Hence, following the previous notations taking $u=1$, for all $\alpha_1,\dots ,\alpha_r\in \res^*$
we write
$$
F_t=\frac{1}{(t+1)!} \sum_{i=1}^5 \alpha_i L_i^{t+1}
$$
$t\ge 1$.
From the last result we have  $I(X)^{\perp}=\langle F_t; t\ge 1\rangle$.
Indeed, we can recover the powers $L_i^w$ from $F_{w+1}$, $w\ge 1$, see \lemref{Lemma2}.
For instance
$$
L_1^w =\beta_w \circ F_{w+1}
$$
with $\beta_w=\frac{1}{\alpha_1(w+2)(w+1)}(x_0(x_0+x_3))$.

On the other hand, $x_3$ is a non-zero divisor of $R/I(X)$ but
$\mathcal F=\{F_t; t\ge 1\}$ is not $G$-admissible with respect to $x_3$ because
$x_3\circ F_1=(\alpha_1-\alpha_5)y_0+\alpha_2y_1+ \alpha_3y_2+(\alpha_1+ \dots+\alpha_5)y_3$, and this is $\neq 0$ since  $\alpha_2,\alpha_3 \neq 0$. Recall that $X$ is not arithmetically Gorenstein.
\end{example}

\medskip
{ { Without  loss of  generality, given a set of distinct points $X=\{P_1,\dots, P_r\}$, after a change of coordinates,   { {we may assume that $x_n(P_i) =1  $   for every $i=1, \dots, r;$ }} that is $P_i=[a_{0,i}, \dots, a_{n-1,i}, 1].$  In particular this means that  we assume  $x_n$ is not zero divisor in $R/I(X) $ and $L_i = a_{0,i}y_0+ \dots+  a_{n-1,i} y_{n-1} +y_n $ is the associated form to $P_i$. Through this section we always keep this assumption, which is not essential, but it simplifies  the presentation. Just to be clear we will specify that
  $X=\{P_1,\dots, P_r\}$ is  a set of distinct points of $\mathbb P^n_{\res}\setminus \{x_n=0\} $  and  when we will write   that $I(X)^{\perp}$ is     $G$-admissible, it will mean  that it  is
  $G$-admissible  with respect to $x_n$.
 }}

\medskip
Recalling   Definition \ref{G-modules},   a family  $\mathcal F=\{F_t; t\ge 1\}$ is $G$-admissible with respect to a linear form  $x_n,$ if the following conditions hold:
\medskip
\begin{enumerate}
\item[(1)] for all $t\in \mathbb N_+$
$$
{x_n} \circ F_t =
\left\{
\begin{array}{ll}
 F_{t-1}   & \text{ if }  t> 1 \ \\
0 & \text{ otherwise.}
\end{array}
\right.
$$
\item[(2)] $\ann_R \langle F_t\rangle \circ F_{t+1}=\langle F_{1}\rangle$ for all $t\in \mathbb N_+$.
\end{enumerate}

\medskip
\begin{remark}   \label{1}
It is easy to check  that if we consider the family of elements     $\{F_t\} $ as defined  in  Proposition \ref{latte}, then ${x_n} \circ F_t= F_{t-1}$  if  $t> 1. $  Hence condition (1) of the
$G$-admissibility for the above family is equivalent to
 $ x_n \circ F_1 =0\ (t=1)$  which is in turn
  equivalent   to $$\sum_{i=1}^r \alpha_i L_i^{u}=0.$$  This last equality implies that $\HF_X(u)= \dim_{\res} \langle L_1^u, \dots, L_r^u \rangle_{\res} < r. $ Since $\HF_X(t)= r $  for all $t \ge s $ where $s$ is the socle degree of $R/I(X),$ then $u\le s-1.$
Notice that if  $\Delta \HF_X(s)=1$ (this is the case for being  $X$   arithmetically Gorenstein), then
the $\res$-vector space generated by
$L_1^{s-1},\dots ,L_r^{s-1}$
has dimension $\HF_X(s-1)=r-1$. From this we get that
$\alpha_1, \dots , \alpha_n$ are unique (up to the multiplication by scalars in $\res^*$).
\end{remark}

Given a set of points, we have a  good candidate  for testing if $X$ is arithmetically Gorenstein or equivalently if $I(X)^{\perp} $ is $G$-admissible  according to Definition \ref{G-modules}. We can prove  the following result.

\medskip
\begin{corollary} \label{candidate} Let $X=\{P_1,\dots, P_r\}$ be a set of distinct points of $\mathbb P^n_{\res}\setminus \{x_n=0\}. $   Let $L_1, \dots, L_r$ be the corresponding associated linear forms and let $s$ be the socle degree of $R/I(X).$ The following conditions are equivalent

\begin{enumerate}
    \item[(a)] $X$ is arithmetically Gorenstein,
    \item[(b)] there exist   $\alpha_1, \dots, \alpha_r \in \res^* $ (unique up to scalars in $\res^*$) such that the $R$-submodule $M$ of $\Gamma$ generated by the elements
$$ F_t:=\frac{1}{(t+s-1)!}\sum_{i=1}^r \alpha_i L_i^{t+s-1} $$
for $t\ge 1, $  is $G$-admissible.
\end{enumerate}
If this is the case, then $I(X)^{\perp}=M.$
\end{corollary}
\begin{proof} We assume  (a),   then   by Theorem \ref{bijecGor} we know that $I(X)^{\perp}$ is
$G$-admissible and  (b)  follows by  Proposition \ref{fat-Geramita} and Remark \ref{1}. Conversely assume  (b),   then by Proposition \ref{latte} we have   $I(X)=\ann{M}$ and hence  (a) follows  by Theorem \ref{bijecGor}.
\end{proof}

\medskip
  Our aim is now   to deepen and simplify condition (2) of $G$-admissibility.
In general we have:

\begin{lemma}
\label{sequence}
Let $X=\{P_1,\dots, P_r\}$ be a set of distinct points of
$\mathbb P^n_{\res}\setminus \{x_n=0\} $.
Let $\{F_t\}_{t\ge 1}$ be a family of elements of $I(X)^{\perp}$ satisfying the condition {\rm{(1)}} of
$G$-admissibility with respect to $x_n$.
Then for all $t\ge 1$ it holds
$$
\langle F_1\rangle \subseteq  \ann(F_1)\circ F_2
\subseteq \dots
\subseteq \ann(F_t)\circ F_{t+1}
\subseteq \ann(F_{t+1})\circ F_{t+2}
\subseteq (I(X)+(x_n))^{\perp}
$$
\end{lemma}
\begin{proof}
Since $F_1=x_n\circ F_2$ we get
$$
F_1= x_n\circ F_2\in \ann(F_1)\circ F_2
$$
so $\langle F_1\rangle \subseteq  \ann(F_1)\circ F_2$.

For all $i\ge 1$ we have
$x_n \ann(F_{i})\subset \ann(F_{i+1})$ since:
$$
(x_n \ann(F_{i}))\circ F_{i+1}= \ann(F_{i})\circ (x_n\circ F_{i+1})=
\ann(F_{i})\circ F_{i}=0.
$$
Hence we obtain
$$
\langle F_1\rangle \subseteq  \ann(F_1)\circ F_2
\subseteq \dots
\subseteq \ann(F_t)\circ F_{t+1}
\subseteq \ann(F_{t+1})\circ F_{t+2}.
$$

Being $F_{t+2}\in I(X)^{\perp}$ we deduce that
$\ann(F_{t+1})\circ F_{t+2} \subseteq I(X)^{\perp}$.
On the other hand
$$
x_n\circ (\ann(F_{t+1})\circ F_{t+2} )=
\ann(F_{t+1})\circ (x_n\circ F_{t+2} )=
\ann(F_{t+1})\circ F_{t+1} =0.
$$
Hence
$\ann(F_{t+1})\circ F_{t+2}\subseteq (x_n)^{\perp} $.
Recall that $(I(X)+(x_n))^{\perp}=I(X)^{\perp} \cap (x_n)^{\perp}$, so
$$\ann(F_{t+1})\circ F_{t+2}
\subseteq (I(X)+(x_n))^{\perp}.$$
\end{proof}

\medskip
Recall that condition $(2)$ of $G$-admissibility for $\mathcal F$ prescribes the equality
$$
\langle F_1\rangle= \ann(F_t)\circ F_{t+1},
$$
for all $t\ge 1$.

\bigskip
In the following example we present a set of points where $I(X)^{\perp}$ is generated by a family $\{F_t\}$ of elements of $\Gamma$ verifying condition $(1)$ of the $G$-admissibility, but not  condition $(2).$

\begin{example}
\label{no-gor}
Consider the points $P_1=[1,0,0,1]$, $P_2=[0,1,0,1]$, $P_3=[0,0,1,1]$, $P_4=[0,0,0,1]$,
$P_5=[1,1,1,1]$ and $P_6=[-1,1,-1,1]$ in $\mathbb P_{\res}^3$.
The socle degree of $R/I(X)$ is $s=2$ and $X$ is not arithmetically Gorenstein, for instance because the $h$-vector is $(1, 3, 2)$.
If we take $\alpha_1=-1, \alpha_2=-3, \alpha_3=-1, \alpha_4=2, \alpha_5=2, \alpha_6=1$ then
the family $\mathcal F$ defined by the polynomials
$$ F_t=\frac{1}{(t+1)!}\sum_{i=1}^6 \alpha_i L_i^{t+1},
$$
$t\ge 1$, satisfies the condition $(1)$ of
$G$-admissibility with respect to $x_3$.
On the other hand, we can verify that
$\ann(F_1)\circ F_2=\langle F_1\rangle, $ but
$\ann(F_2)\circ F_3  \supsetneq\langle F_1\rangle$ (see the computation below where $A,$ or $B,$ do not belong to  $\langle F_1\rangle$).
Hence the family $\mathcal F$ does not satisfies the condition $(2)$ of $G$-admissibility and indeed $X$ is not Gorenstein.
Notice that
$$
\ann(F_2)\circ F_3 =(I(X)+(y_3))^{\perp}=\langle
A:=y_0y_1+y_0y_2+y_1y_2, B:=y_0^2+2y_0y_2+y_2^2\rangle.
$$
and
$$F_1=\frac{1}{2}(y_0^2+y_0y_1+3y_0y_2+y_1y_2+y_2^2)=
\frac{1}{2}(A+B).
$$
\end{example}

\bigskip
The following  result proves that the chain of \lemref{sequence} stabilizes at $(I(X)+(x_n))^{\perp}.$

\medskip
\begin{proposition}
\label{cafe}
Let $X=\{P_1,\dots, P_r\}$ be a set of distinct points of $\mathbb P^n_{\res}\setminus \{x_n=0\} $  of socle degree $s$ and let $L_1, \dots, L_r$ be the associated linear forms. Let  $I(X)^{\perp}$ be generated by the family $\mathcal F=\{F_t; t\ge 1\}$ where
$$ F_t:=\frac{1}{(t+s-1)!}\sum_{i=1}^r \alpha_i L_i^{t+s-1}$$
with $\alpha_1,\dots ,\alpha_r\in \res^*.$
Assume that the family $\mathcal F=\{F_t; t\ge 1\}$ satisfies the condition $(1)$ of $G$-admissibility.
Then for all $t\ge s$
$$
\ann(F_t)\circ F_{t+1} =(I(X)+(x_n))^{\perp}.
$$
\end{proposition}
\begin{proof}
Since $\mathcal F$ satisfies the condition $(1)$ of $G$-admissibility with respect to $x_n$ we have
$$
\langle F_1\rangle
\subset \ann(F_t)\circ F_{t+1}
\subset \ann(F_{t+1})\circ F_{t+2}
\subset (I(X)+(x_n))^{\perp}
$$
for all $t\ge 1$.
We only need to prove that
$$
(I(X)+(x_n))^{\perp}
\subset
\ann(F_s)\circ F_{s+1}.
$$
Let $G\in (I(X)+(x_n))^{\perp}$ be an homogeneous  form.
Since $(I(X)+(x_n))^{\perp}_t=0$ for all $t\ge s+1$,
we may assume that $w=\deg(G)\le s$.

From \lemref{Lemma2} for all $i=1,\dots ,r$ there is an homogeneous form $\beta_i$ of degree
$2s-w$
such that
$$
L_i^w= \beta_i\circ F_{s+1}.
$$
Since $G\in I(X)^{\perp}_w$ we have
$$
G=\sum_{i=1}^r \lambda_i L_i^w=
(\sum_{i=1}^r \lambda_i\beta_i)\circ F_{s+1}
$$
where $\lambda_i\in\res$,
we write $\delta= \sum_{i=1}^r \lambda_i\beta_i$.
Since $x_n\circ G=0$ we get
$$
0=x_n\circ G=\delta \circ F_s,
$$
so $\delta \in \ann(F_s)$ and then
$G=\delta \circ F_{s+1}\in \ann(F_s)\circ F_{s+1}$.
\end{proof}

\medskip
We are ready to rephrase the $G$-admissibility for the inverse system of a   set of distinct points $X=\{P_1,\dots, P_r\}$ of $\mathbb P^n_{\res} $  in terms of a finite number of conditions.

\medskip
 \begin{proposition}[$G$-admissibility for points]
  \label{G-points} Let $X=\{P_1,\dots, P_r\}$ be a set of distinct points of $\mathbb P^n_{\res}\setminus \{x_n=0\} $ of socle degree $s$ and let $L_1, \dots, L_r$ be the associate linear forms. Given $\alpha_1,\dots ,\alpha_r\in \res^*$ consider the   $R$-submodule $M$ of $\Gamma$  generated by the elements
 $$   F_t:=\frac{1}{(t+s-1)!}\sum_{i=1}^r \alpha_i L_i^{t+s-1}     $$
 Then $ M= I(X)^{\perp}$ is $G$-admissible (and equivalently  $X$ is arithmetically Gorenstein) if and only if the following conditions hold:
 \begin{enumerate}
    \item[(1)] $\sum_{i=1}^r \alpha_i L_i^{s-1}=0$,
    \item[(2)] $\langle F_1\rangle =\ann(F_s)\circ F_{s+1}$.
 \end{enumerate}
  \end{proposition}
 \begin{proof}
 If $M$ is $G$-admissible, then $(1)$ and $(2)$ follows by the definition,  because $\langle F_1\rangle =\ann(F_t)\circ F_{t+1}$ for all $t \ge 1, $ in particular for $t=s,$ and   $x_n \circ F_1=0. $
 Assume now $(1)$ and $(2)$.
 As a consequence of the Proposition \ref{cafe},   condition $(2)$ of $G$-admissibility can be checked just in the case $t=s$, i.e. condition $(2)$ in Definition \ref{G-modules}
 is equivalent to
 $\langle F_1\rangle =\ann(F_s)\circ F_{s+1}$. Concerning (1), it follows by Remark \ref{1}.
   \end{proof}

\medskip
 Next result   answers one of the main question of this paper.   Notice that we characterize the inverse system of Gorenstein points (a not finitely generated $R$-module!)   in terms of one single  polynomial of $\Gamma$ by analogy with the Artinian case.   For this reason the result can be considered as a strengthening of \cite[Theorem 2.2]{Toh14}.

\medskip
\begin{theorem}
\label{char-points-gor}
Let $X=\{P_1,\dots, P_r\}$ be a set of points of $\mathbb P^n_{\res}\setminus \{x_n=0\} $ of socle degree $s$ and let $L_1, \dots, L_r$ be the associate linear forms.
Then the following conditions are equivalent:
\begin{enumerate}
    \item $X$ is arithmetically Gorenstein,
    \item there exists a polynomial $F=\sum_{i=1}^r \alpha_i L_i^{s}$ with
     $\alpha_1,\dots, \alpha_r\in \res$  such that
    \begin{enumerate}
        \item[(a)] $\sum_{i=1}^r \alpha_i L_i^{s-1}=0$,
        \item[(b)] $\dim_{\res}\langle F\rangle \ge r.$
    \end{enumerate}
\end{enumerate}

\noindent
If this is the case,  then $\alpha_1,\dots,
\alpha_r\in \res^*$ are uniquely  determined (up to scalars in   $\res^*$),
$I(X)^\perp =\langle F_t; t\ge 1\rangle$
where $F_t=\frac{1}{(t+s-1)!}\sum_{i=1}^r \alpha_i L_i^{t+s-1}$, in particular $F_1=\frac{1}{s!} F$ and
the family $\{F_t;t\ge 1\}$ is $G$-admissible.
\end{theorem}
%{\red{ b) superfluo se mettiamo $\alpha_i \in k^*$?}}
\begin{proof}
Assume that $X$ is arithmetically Gorenstein.
Then by Corollary \ref{candidate} there exist $\alpha_1, \dots, \alpha_r \in \res^* $ such that the  $R$-submodule $M$ of $\Gamma$ generated by  $F_t=\frac{1}{(t+s-1)!}\sum_{i=1}^r \alpha_i L_i^{t+s-1}$ for $t \ge 1$ is $G$-admissible and $I(X)^{\perp}=M. $
It is enough to take $F=s! F_1$.  In fact by Proposition \ref{G-points} we have $(a)$ and  the equality $\langle F_1 \rangle = \langle F \rangle =  \ann(F_s)\circ F_{s+1}.$ Hence by Proposition \ref{cafe}   we have $\langle F_1 \rangle = \langle F \rangle =  \ann(F_s)\circ F_{s+1} = (I(X) + (x_n)^{\perp}. $ Now $R/I(X) + (x_n)$ is an Artinian ring since $x_n$ is a non-zero divisor in $R/I(X)$, hence of finite dimension as $\res$-vector space. Its length equals the dimension and
$\dim_{\res} R/(I(X) + (x_n))= r. $
Hence $\dim_{\res}\langle F\rangle = \length (R/(I(X)+(x_n))) =  r.$

Conversely, by Proposition \ref{G-points} it is enough to prove that $ \langle F_1 \rangle =  \ann(F_s)\circ F_{s+1}.$ This is clear because $\langle F_1\rangle \subseteq   \ann(F_s)\circ F_{s+1}$ and by  $(b)$
$$\dim_{\res}\langle F\rangle \ge r =
\length (R/I(X)+(x_n))= \length ( \ann(F_s)\circ F_{s+1}).$$
The last equality follows again by Proposition \ref{cafe}.
\end{proof}

\medskip
\begin{remark}
\label{Pi} Let $X=\{P_1,\dots, P_r\}$ be a set of points of $\mathbb P^n_{\res}  $ of socle degree $s$ and let $L_1, \dots, L_r$ be the associate linear forms.   Denote  $P_i=[a_{0,i}, \dots, a_{n,i}]$ for $i=1, \dots, r.$
Along the paper we assumed for simplicity  that $x_n$ is a non-zero divisor of $R/I(X), $ that is $x_n(P_i) =a_{n,i} \neq 0$ for all $i=1, \dots,r. $  In particular we   assumed $a_{n,i}=1, $    hence $x_n \circ L_i=1 $ for all $i=1, \dots, r$.

In general, if  $z=b_0 x_0+\dots+b_n x_n$ is any linear form in $R$ such that $z$ is a regular element of  $R/I(X), $  it means that  $z(P_i)= z \circ L_i=  b_0 a_{0,i}+\cdots+b_n a_{n,i}  \neq 0  $  for all $i=1, \dots,r.$ We remark  that $(1)$  and $(2) $ in Theorem \ref{char-points-gor} can be rephrased  as   follows:
\begin{enumerate}
    \item $X$ is arithmetically Gorenstein,
    \item there exists a polynomial $F=\sum_{i=1}^r {\frac{\alpha_i}{z(P_i)}} L_i^{s}$ with
     $\alpha_1,\dots, \alpha_r\in \res$  such that
    \begin{enumerate}
        \item[(a)] $\sum_{i=1}^r \alpha_i L_i^{s-1}=0$,
        \item[(b)] $\dim_{\res}\langle F\rangle \ge r.$
    \end{enumerate}
    \end{enumerate}

    \noindent
    If this is the case,  then $\alpha_1,\dots,
\alpha_r\in \res^*$ are uniquely  determined, modulo $\res^*$,
$I(X)^\perp =\langle F_t; t\ge 1\rangle$
where $F_t=\frac{1}{(t+s-1)!}\sum_{i=1}^r {\frac{\alpha_i}{z(P_i)^{t}}} L_i^{t+s-1}$, in particular $F_1=\frac{1}{s!} F$ and
the family $\{F_t;t\ge 1\}$ is $G$-admissible with respect to $z$.
\end{remark}

%%%%%%%%%%%%%%%%%%%%%%%%%%%%%%%%%%%%%%%%%%%%%%%%%%%%%%%%%%%%%%%%%%%%%%%%%%%%%%%%%%%%%%%%%%%%%%%%%
%%%%%%%%%%%%%%%%%%%%%%%%%%%%%%%%%%%%%%%%%%%%%%%%%%%%%%%%%%%%%%%%%%%%%%%%%%%%%%%%%%%%%%%%%%%%%%%%%
%%%%%%%%%%%%%%%%%%%%%%%%%%%%%%%%%%%%%%%%%%%%%%%%%%%%%%%%%%%%%%%%%%%%%%%%%%%%%%%%%%%%%%%%%%%%%%%%%
%%%%%%%%%%%%%%%%%%%%%%%%%%%%%%%%%%%%%%%%%%%%%%%%%%%%%%%%%%%%%%%%%%%%%%%%%%%%%%%%%%%%%%%%%%%%%%%%%
%%%%%%%%%%%%%%%%%%%%%%%%%%%%%%%%%%%%%%%%%%%%%%%%%%%%%%%%%%%%%%%%%%%%%%%%%%%%%%%%%%%%%%%%%%%%%%%%%
\bigskip
\section{Effective construction of Gorenstein Points}
Using \thmref{char-points-gor} and \remref{Pi}, in this section we give an explicit description of  the Zarisky locally closed locus of  the   points
$X=\{P_1,\dots, P_r\}\subset  \mathbb P^n_{\res}$  which are arithmetically Gorenstein.

Given a linear form  $z=b_0 x_0+\cdots+b_n x_n$ in $R$ and $L=a_0 y_0+\cdots+a_n y_n$ in $\Gamma,$  we denote by $z\cdot L$ the   scalar    $z\cdot L:=b_0a_0+\cdots b_na_n= z \circ L$.
Notice that if $L$ is the linear form associated to a point $P\in \mathbb P_{\res}^n, $ then
$z\cdot L= z(P)$.
%Hence, given a set of points  $X=\{P_1,\dots, P_r\}\subset  \mathbb P^n_{\res}$,a linear form $z$ defines a non-zero divisor of a $R/I(X)$  if and only if $z\cdot L_i=z(P_i)\neq 0$ for $i=1,\dots r$, where $L_i$ is the linear form defined by $P_i$.

\medskip
\begin{definition} \label{W}
Given integers $2 \le s\le r$, we define  the set $W_{r,s}$ of linear forms
$\{L_1,\dots,L_r\}$ in $\Gamma$ satisfying the following conditions:
\begin{enumerate}
       \item $\rank(L_1^{s-1},\dots,L^{s-1}_r)=  \rank(L_1^{s-1},\dots,\widehat{L^{s-1}_i},\dots,L^{s-1}_r)=  r-1$, , $i=1,\dots,r.$
    %\item $\rank(L_1^{s-1},\dots,\widehat{L_i^{s-1}},\dots, L^{s-1}_r)= r-1$, $i=1,\dots,r$,
    \item let      $\alpha_1,\dots,\alpha_r\in \res^*$  the elements uniquely determined (up to scalars in   $\res^*$), by  the identity $ \sum_{i=1}^r \alpha_i L_i^{s-1}=~0   $  given by $(1).$
    Let $z$ be a linear form such that $z\cdot L_i\neq 0$ for $i=1,\dots,r$.
    If $F=\sum_{i=1}^r \frac{\alpha_i}{z\cdot L_i} L_i^s, $ then $\dim_{\res}\langle F \rangle\ge r.$
\end{enumerate}
\end{definition}
\vskip 2mm
Assume $L_1, \dots, L_r $ linear forms associated to $X=\{P_1, \dots, P_r\}, $ we remark that (1) in the above definition is equivalent $X$ having the Cayley-Bacharach property, for details see \cite{DGO83} and the proof of Theorem 20 in \cite{Toh14}.

\medskip
\begin{remark}
\label{zariski} We remark that $\mathbb P(W_{r,s})$ defines a Zariski locally closed subset of $(\mathbb P_{\res}^n)^r$   described as it follows.
Let $ L_1,\dots,L_r $  be linear forms in $\Gamma= \res[y_0,\dots, y_n] $  and consider the base $B$ of the monomials of degree $s-1$ in $y_0,\dots,y_n$.
Let $M_{r,s}$ be the  $\binom{s-1+n}{n}\times r$ matrix whose the  columns are  respectively the coefficients   of $L_1^{s-1},\dots,L^{s-1}_r$ with respect to the base $B.$ Hence the columns of  $M_{r,s}$ are the monomials of degree $s-1$ in the coefficients of the $r$ linear forms.

The condition $\rank(L_1^{s-1},\dots,L^{s-1}_r)\le  r-1$ is satisfied by choosing the coefficients of the linear forms in      the irreducible Zariski closed  set  defined by the vanishing of all $r \times r$ minors of $M_{r,s}$, say $V_{r,s-1}. $ Consider now a  submatrix $N$ of $M_{r,s}$ defined by $r-1$ rows of $M_{r,s}$. The condition $\rank(L_1^{s-1},\dots,\widehat{L^{s-1}_i},\dots,L^{s-1}_r)\ge  r-1$ is satisfied if { {all}} the  $r-1 \times r-1$  minors $N_i$  of $ N $ do not vanish. Set  $\alpha_i := (-1)^i N_i \neq 0, $ then   $\dim_{\res}\langle F \rangle\ge r $ imposes an other open condition on the coefficients of the linear forms.
\end{remark}

\medskip
\begin{proposition}
\label{steGor}
Let  $s\le r$.
A set of points $X=\{P_1,\dots, P_r\}\subset  \mathbb P^n_{\res}$ is arithmetically Gorenstein with socle degree $s$ if and only if $(L_1,\dots,L_r)\in \mathbb P(W_{r,s})$ where $L_i$ is the linear form defined by $P_i$, $i=1,\dots,r$.
\end{proposition}
\begin{proof}
Let $X= \{P_1,\dots, P_r\}\subset  \mathbb P^n_{\res}$ be a   Gorenstein set of points with socle degree $s$.
Let $z\in \Gamma$ be a linear form defining a non-zero divisor of $R/I(X)$.

From \thmref{char-points-gor}, see \remref{Pi}, there exists a polynomial
$F=\sum_{i=1}^r\frac{\alpha_i}{z\cdot L_i} L_i^{s}$ with
 $\alpha_1,\dots,\alpha_r\in \res^*$ uniquely determined, up to scalars in $\res^*$, such that
 $H_1=\frac{1}{s!}F$ is a generator of $(I(X)+(z))^{\perp}$.
 Hence, \propref{fat-Geramita},
 $$r-1=\HF_X(s-1)=\dim_{\res}\langle L_1^{s-1},\dots,L_r^{s-1} \rangle=
 \rank\langle L_1^{s-1},\dots,L_r^{s-1} \rangle,$$
i.e. we get condition $(1)$ of Definition \ref{W}.
Condition $(2)(b)$ of \remref{Pi}  implies $(2)$. Hence $(L_1,\dots,L_r)\in \mathbb P(W_{r,s}).$

Conversely, let $(L_1,\dots,L_r)\in \mathbb P(W_{r,s})$ be $r$ linear forms satisfying conditions $(1)$ and $(2)$.
Let $P_i$ be the point defined by $L_i$, $i=1,\dots,r$.
Notice that $z$ defines a non-zero divisor of $R/I(X)$.
Let $\alpha_1,\dots,\alpha_r\in \res^*$  be the elements uniquely determined, up to scalars in $\res^*$, by condition $(1)$, i.e. $\sum_{i=1}^r\alpha_i L^{s-1}=0$.
If we write $F=\sum_{i=1}^r \frac{\alpha_i}{z \cdot L_i} L_i^s$ we have $z\circ F=s(\sum_{i=1}^r\alpha_i L_i^{s-1})=0$, so we get condition $(2)(a)$ of \remref{Pi}.
Condition $(2)$ implies   $(2)(b)$ of \remref{Pi}.

We have to prove that the socle degree of $X$ is $s$ to conclude that $X$ is arithmetically Gorenstein.
From \propref{fat-Geramita} we deduce that $F\in I(X)^{\perp}\cap (z)^{\perp}=(I(X)+(z))^{\perp}$.
Hence, since $z$ is a non-zero divisor of $R/I(X)$, we get
$$
r\le \dim_{\res}\langle F \rangle\le \dim_{\res}(I(X)+(z))^{\perp}=r
$$
and then $\langle F \rangle=(I(X)+(z))^{\perp}$.
Being $F$ a form of degree $s$ we conclude that the socle degree of $X$ is $s$.
\end{proof}

\begin{remark}
Since $\mathbb P(W_{r,s})$ is a Zariski locally closed set intersection of the Zariski closed set $\mathbb P(V_{r,s-1})$ and an  open set, see  \remref{zariski},
if $\mathbb P(W_{r,s})\neq\emptyset$ then a generic point of $\mathbb P(V_{r,s-1})$ defines an arithmetically Gorenstein set
$X\subset \mathbb P^n_{\res}$
of socle degree $s$.
\end{remark}

\vskip 3mm
\begin{example}
Let us consider the set $X$ of the  $4$ points  $P_1=[1,0,0,0]$, $P_2=[0,1,0,0]$, $P_3=[0,0,1,0]$, $P_4=[0,0,0,1]$.
We want to compute the points $P_5=(\alpha,\beta,\gamma,\delta)$  such that $X=\{P_1,\dots,P_5\}$ is arithmetically Gorenstein.

In this case the socle degree is $s=2, $   $L_1=y_0$, $L_2=y_1$, $L_3=y_2$, $L_4=y_3$ and
$L_5=\alpha y_0+\beta y_1+\gamma y_2+\delta y_3$.
We  impose that $(L_1,\dots, L_4, L_5) \in \mathbb P(W_{5,2}).$

The first part of the condition $(1)$ of  \defref{W} is satisfied because the matrix of the coefficients of $L_1,\dots,L_5$

$$
M=
\begin{pmatrix}
1 & 0& 0& 0& \alpha\\
0 & 1& 0& 0& \beta\\
0 & 0& 1& 0& \gamma\\
0 & 0& 0& 1& \delta
\end{pmatrix}
$$

\noindent
is of maximal rank $r-1=4$.
The maximal minors of $M$ that are: $-\alpha, - \beta, -\gamma, - \delta , 1$,
so $\alpha_1,\dots, \alpha_5$ are equal to $-\alpha, - \beta, -\gamma,  - \delta, 1$, up to scalars in $\res^*$.
The second part of condition $(1)$ of  \defref{W} is equivalent to $\alpha,\beta,\gamma, \delta \in \res^*$.

The second condition of \defref{W} is always satisfied.
Let $z=a_0 x_0+\cdots+ a_3 x_3$ be a linear form such that $z\cdot L_i\neq 0$ for $i=1,\dots , 5$.
Let us write  the degree two form in condition $(2)$ of  \defref{W}
$$
F=\frac{-\alpha}{a_0}y_0^2+\frac{-\beta}{a_1}y_1^2+\frac{-\gamma}{a_2}y_2^2+\frac{-\delta}{a_3}y_3^2+
\frac{1}{a_0 \alpha+a_1 \beta+a_2 \gamma+a_3\delta}L_5^2.
$$
Recall that  $\langle F \rangle=\langle F, x_0\circ F, x_1\circ F,x_2\circ F,x_3\circ F, 1 \rangle_{\res}$;
a simple computation shows that
$\dim_{\res}\langle x_0\circ F, x_1\circ F,x_2\circ F,x_3\circ F \rangle_{\res}$ agrees with the rank of the following matrix of the coefficients of the linear forms $x_0\circ F, x_1\circ F,x_2\circ F,x_3\circ F $

$$
\begin{pmatrix}
 -a_1 \beta-a_2 \gamma-a_3 \delta & a_0 \beta & a_0 \gamma & a_0 \delta \\
 \alpha a_1 & -a_0 \alpha-a_2 \gamma-a_3 \delta & a_1 \gamma & a_1 \delta \\
 \alpha a_2 & \beta a_2 & -a_0 \alpha-a_1 \beta-a_3 \delta & a_2 \delta \\
 \alpha a_3 & \beta a_3 & \gamma a_3 & -a_0 \alpha-a_1 \beta-a_2 \gamma \\
 \end{pmatrix}
$$

\noindent
The rank of this matrix is always three because its determinant is zero and the minor of the last three columns and  rows is non-zero:
$$
a_0\alpha(a_0\alpha+a_1\beta+a_2\gamma+a_3\delta)^2=
\alpha (z\cdot L_1)(z\cdot L_5)^2\neq0.
$$
Hence  we get that $\dim_{\res}\langle F \rangle=5$, i.e. we get the condition $(2)$ of \defref{W}.

Hence the set of points $P_5=(\alpha,\beta,\gamma,\delta)\in  \mathbb P^3_{\res}$ such that $X=\{P_1,\dots,P_5\}$ is arithmetically Gorenstein is the open
set of points $P_5\in \mathbb P_{\res}^3$ such that $\alpha\neq 0, \beta\neq 0, \gamma\neq 0, \delta \neq 0.$
In other words: $X$ is arithmetically Gorenstein if and only if  $P_5$ do not belong to any of the  four  hyperplanes  defined by the coordinate
simplex $\{[1,0,0,0], [0,1,0,0], [0,0,1,0], [0,0,0,1]\}$ of $\mathbb P_{\res}^3   $.
In particular we remark that if $P_5=[-1,0,0,1]$, see also \exref{DGO-example}, then $X$ is not Gorenstein.
\end{example}

\begin{example}
 Let $Y_1$ be the set of $6$ points of $\mathbb P^3_{\res}$
\begin{multline*}
    Y_1=\{P_1=[1,0,0,0], P_2=[0,1,0,0], P_3=[0,0,1,0], \\P_4=[0,0,0,1], P_5=[1,1,1,1], P_6=[-5,1,-3,1]\}.
\end{multline*}

\noindent
We want to find a set of points $Y_2=\{P_7=[a,b,c,d ], P_8=[A,B,C,D]\}$ such that $X=Y_1\cup Y_2$ is arithmetically Gorenstein with h-vector $\{1,3,3,1\}$.
This means that the schemes $Y_1$ and $Y_2$ are Gorenstein linked by $X$.

Let $L_i$ be the associated linear form of $P_i$, $i=1,\dots,r=8$.
Let us consider the $10\times8$ matrix $M$ defined by $L_1^2,\dots,L_8^2$; the columns  of the matrix corresponds to the coefficients of the degree two monomials on $y_0,\dots,y_3$ ordered lexicographically
(i.e. $y_0^2,2y_0y_1,2 y_0y_2,2 y_0y_3, y_1^2,$ $  2 y_1y_2,2 y_1y_3, y_2^2, 2 y_2y_3,y_3^2$)

$$
M=
\begin{pmatrix}
1 & 0& 0& 0& 1&25&a^2&A^2\\
0 & 0& 0& 0& 1&-5& ab& AB\\
0 & 0& 0& 0& 1&15&ac& AC\\
0 & 0& 0& 0& 1&-5&ad& AD\\
0 & 1& 0& 0& 1&1&b^2&B^2\\
0 & 0& 0& 0& 1&-3&  bc& BC\\
0 & 0& 0& 0& 1&1& bd& BD\\
0 & 0& 1& 0& 1&9&c^2&C^2\\
0 & 0& 0& 0& 1&-3& cd& CD\\
0 & 0& 0& 1& 1&1&d^2&D^2
\end{pmatrix}.
$$

\noindent
Condition $(1)$ of \remref{W} is equivalent to $\rank(M)=7$ and that every sub-matrix of $M$ deleting any of its columns is of rank $7$ as well.

The variety $Z$ of pairs $(P_7,P_8)\in (\mathbb P_{\res}^3)^2$ vanishing all $8\times 8$ minors of $M$ is defined by a radical ideal minimally generated by $6$ bi-homogeneous forms of bi-degree $(2,2)$ in the set of variables $\{a,b,c,d\}$ and $\{A,B,C,D\}$.
A standard computation shows that the dimension of $Z$ is $3$ and its degree $8$.
We can   deduce this fact also by geometric arguments. Indeed, given a generic point $P_7$ the degree two component $I(\{P_1,\dots,P_7\})_2$ is generated by three quadrics containing the set
$\{P_1,\dots,P_7\}$ and an extra point $P_8$.
 Hence $P_8$ is determined by $P_7$.
The whole set $\{P_1,\dots,P_8\}$ is Gorenstein since it is the complete intersection defined by the three quadrics.

Let us consider the points $Y_2=\{[-2,0,5,1], [-2, 30/13,5,19]\}$ then we get that the matrix $M$ satisfies
condition $(1)$ of \remref{W}.
 The $\alpha$'s  are
 $\alpha_1=7347, \alpha_2=6975/13, \alpha_3=12555, \alpha_4=-8277, \alpha_5=- 465, \alpha_6=-210, \alpha_7=-434, \alpha_8=26$,  unique up to scalars in $\res^*$.
  Let us consider the non-zero divisor $z=x_0+x_1+x_2+x_3$.
 Then   the  degree three form $F$ of \defref{W} is
$$
F=7347L_1^3+\frac{6975}{13}L_2^3+12555L_3^3-8277L_4^3-\frac{465}{4}L_5^3+35L_6^3-\frac{217}{2}L_7^3+\frac{169}{158}L_8^3.
$$
A simple computation shows that that
$\dim_{\res}\langle F\rangle \ge 8$.
Hence $X=Y_1\cup Y_2$ is arithmetically Gorenstein with h-vector $(1,3,3,1)$, according to  \propref{steGor}.
\end{example}

%%%%%%%%%%%%%%%%%%%%%%%%%%%%%%%%%%%%%%%%%%%%%%%%%%%%%%%%%%%%%%%%
%%%%%%%%%%%%%%%%%%%%%%%%%%%%%%%%%%%%%%%%%%%%%%%%%%%%%%%%%%%%%%%%
%%%%%%%%%%%%%%%%%%%%%%%%%%%%%%%%%%%%%%%%%%%%%%%%%%%%%%%%%%%%%%%%
%%%%%%%%%%%%%%%%%%%%%%%%%%%%%%%%%%%%%%%%%%%%%%%%%%%%%%%%%%%%%%%%
%%%%%%%%%%%%%%%%%%%%%%%%%%%%%%%%%%%%%%%%%%%%%%%%%%%%%%%%%%%%%%%%
%%%%%%%%%%%%%%%%%%%%%%%%%%%%%%%%%%%%%%%%%%%%%%%%%%%%%%%%%%%%%%%%
\bigskip
\section{Artinian reduction of Gorenstein points}

Let $R=\res [x_0,\dots, x_n]$ be and let $B=R/J $ be  an Artinian Gorenstein graded $\res$-algebra of  length $r$ and socle degree $s.  $   The following  natural question comes up: under which conditions an Artinian Gorenstein ring is the Artinian reduction of a Gorenstein reduced set $X$ of points in $\mathbb P^n_{\res}$?
%We will see, for instance, that   any Artinian generic Gorenstein algebra with Hilbert function $\{1,11,11,1\}$ is not the Artinian reduction of   a Gorenstein zero-dimensional reduced scheme.
Since $J^{\perp} = \langle F \rangle $ with $F\in \Gamma $ of degree $s, $ the question can be rephrased in terms of $F.$

 This was the main aim of the paper by  Tohaneanu \cite{Toh14}.   A similar problem was studied in the work of Cho and Iarrobino (\cite[Proposition 1.13 and Theorem 3.3]{CI12}) in the less restrictive situation of a general zero-dimensional scheme.
 Boij in \cite{Boi99} describes  the possible Gorenstein Artinian quotients of the coordinate ring of a set of points in projective space through the corresponding canonical module.

Observe that condition $(2)$ of \thmref{char-points-gor} establishes some restrictions on a form $F$ of degree $s$ to be a generator of the inverse system of an Artinian reduction of $ R/I(X).$
On the other hand, \thmref{bijecGor} and the definition of $G$-admissible submodules of $\Gamma$ are based in the lifting of a minimal reduction of the one-dimensional Gorenstein ring $R/I$  with respect to a linear non-zero divisor $z$. In fact if   $M=\langle H_l, l\in \mathbb N_+\rangle$ is a $G$-admissible $R$-submodule of $\Gamma, $  equivalently it is the inverse system of a Gorenstein ring $R/I  $ of socle degree $s,$ then $ (I +(z))^{\perp} = \langle H_1\rangle. $

As a consequence of Theorem \ref{char-points-gor}, in the next result we give a complete answer to our question, see also \cite[Theorem 2.2]{Toh14}.

\medskip
\begin{proposition} \label{reductionPoints}
Let $X=\{P_1,\dots, P_r\}  \subseteq \mathbb P^n_{\res}\setminus \{x_n=0\} $ be a set of   Gorenstein points   of socle degree $s$ and let $L_1, \dots, L_r$ be the associate linear forms.
Let $B=R/J$ be a Gorenstein Artin algebra of length $r$ and socle degree $s$ and let  $J^{\perp}=\langle F \rangle.$

The following conditions are equivalent:
\begin{enumerate}
    \item $B$ is the Artinian reduction of $X$ with respect to $x_n$,
    \item   there exist   $\alpha_1,\dots, \alpha_r\in \res^*$ (unique up scalars) such that $F= \sum_{i=1}^r \alpha_i L_i^{s}$   \ \  and $\sum_{i=1}^r \alpha_i L_i^{s-1}=0. $
\end{enumerate}
\end{proposition}
\begin{proof}
Assume that $B$ is the  Artinian reduction of $X$ with respect to $x_n, $ then  $I(X)+(x_n) =J.$  Let $\{H_t; t\ge 1\}$ be a $G$-admissible generating system of $I(X)^{\perp}$ as in Theorem \thmref{char-points-gor}.
Since $\langle H_1 \rangle = (I(X)+(x_n))^{\perp}$, then $ F=s! H_1$ satisfies $(2)$.

Assume  there exist  $\alpha_1,\dots, \alpha_r\in \res^*$ such that $F=\sum_{i=1}^r \alpha_i L_i^{s}$ and  $\sum_{i=1}^r \alpha_i L_i^{s-1}=0. $ Hence  $x_n\circ F= \sum_{i=1}^r \alpha_i L_i^{s-1}=0 $ and by \propref{fat-Geramita} $F \in I(X)^{\perp}. $
Then
$$
J^{\perp}=\langle F\rangle\subset I(X)^{\perp}\cap (x_n)^{\perp}=(I(X)+(x_n))^{\perp}.
$$
Hence $I(X)+(x_n) \subseteq J $ and they are Gorenstein ideals with the same socle $s.$  By the socle lemma by Kustin and Ulrich, \cite{KU92}, it follows  $J=I(X)+(x_n)$, then  $B=R/J$ is the Artinian reduction of $X$ with respect to $x_n$.
\end{proof}

\medskip
It was   known that there  are Artinian Gorenstein  algebras which are not  Artinian reductions  of Gorenstein points.   As J. Jelisiejew mentioned to us in a private communication, every Artinian reduction of the defining ring of a set of points is smoothable. Hence the non-smoothabillity of a general Artinian algebras      implies that it is not such a reduction. For instance in the case of the Artinian Gorenstein  algebras with socle degree $s=3$,  we may use the results on non-smoothability of $(1,n,n,1)$ for $n \ge 6,$  known in a range since Iarrobino-Kanev,  but proved recently  by   Szafarczyk in \cite{Sza22}.
By following these results the first case of non-smoothability with  $s=3$ is for $n=8.$

\vskip 2mm
Proposition \ref{reductionPoints} allows to  find new examples of Artinian Gorenstein  algebras which are not   Artinian reductions  of Gorenstein points (hence non-smoothables).

\begin{remark}
\label{How}
Let $X =\{P_1,\dots,P_r\} \subset \mathbb P_{\res}^n\setminus\{x_n=0\} $
be a set of   Gorenstein points of socle degree $s$.
Let $L_i=a_{0, i}y_0+\cdots+a_{n,i}y_n\in \Gamma=\res[y_0,\dots,y_n]$ be the associate linear form to $P_i=[a_{0, i},\cdots,a_{n,i}]$, $i=1,\dots,r$.
By \corref{candidate} we know that $(I(X)+(x_n))^{\perp}=\langle F\rangle$
where  $F=\sum_{i=1}^{r} \alpha_i L_i^{s}$, with $\alpha_i\neq 0$.
We write $L_i=A_i+a_{n,i}y_n$, $i=1,\dots, r,$ with  $A_i= a_{0,i}y_0+\cdots+a_{n-1,i}y_{n-1}$.
Since $x_n\circ F=0, $
$F$ is a homogeneous  form in  $\Gamma  $   and then  $F= \sum_{i=1}^{r} \alpha_i A_i^{s} .$

\end{remark}
\medskip

Hence, by the above remark,  a necessary condition for  a homogeneous form $F $ of degree $s$ in $n$ variables  for being the generator of the inverse system of an Artinian  reduction of a set of $r$ Gorenstein points     in $\mathbb P^n_{\res}$    is to decompose   as a sum of $r$ $s$-powers of linear forms in $n$ variables. This fact is related to the so-called Big Waring problem, see for more details  \cite{Ger96}, \cite{VallaSCA96}.

\medskip
\begin{definition}
For all $s\ge 2$, define
$G_n(s)$ as  the least integer   such that the generic form of degree $s$   in $n$ variables  can be written as the sum of $G_n(s)$ $s$-powers of linear forms.
\end{definition}

\medskip
Assume $\res$ algebraically closed. For instance, it is not difficult to see that  $G_n(2)=n. $ This because each  generic quadratic form   may be  associated to a symmetric matrix with non-zero eigenvalues.   The computation of such integer $G_n(s)$ is known as Big Waring problem and it is related to the Hilbert function of generic fat points.  Alexander and Hirschowitz computed the integer $G_n(s)$ in \cite{AH00} proving the following result for a generic form of degree $s$ in $n$ variables.

\medskip
\begin{proposition}
\label{AlxHir}
For every $s\ge 3$ we have
$$
G_n(s)=\min \left\{t  \Bigm| n t\ge \binom{n +s-1 }{s}\right\},
$$
except for the following four cases:
\begin{enumerate}
\item $s=3$, $n=5$ where $G_5(3)=8$, instead of $7$,
\item $s=4$, $n=3$ where $G_3(4)=6$, instead of $5$,
\item $s=4$, $n=4$ where $G_4(4)=10$, instead of $9$,
\item $s=4$, $n=5$ where $G_5(4)=15$, instead of $14$.
\end{enumerate}
\end{proposition}

\medskip
\noindent
Observe that in all previous cases we have
$$
G_n(s)\ge \frac{1}{n}\binom{n +s-1}{s}.
$$
\medskip

\noindent
Given an integer $s=2t+1$ (resp. $s=2t$)  it is well known that the generic form of degree $s$ in  $n$ variables defines a compressed Gorenstein Artin algebra  of length $\rho_n(s)=2\binom{n+t}{t}$ (resp. $\rho_n(s)=2\binom{n+t-1}{t-1}+\binom{n-1+t}{t}$), see  \cite{Iar84}.

\medskip
\begin{proposition}
\label{NoLifGor}
If $s= 2t +1 \ge 3$  and
$$
\frac{1}{n}\binom{2t+n}{2t+1}> 2 \binom{n+t}{t},
$$
or $s=2t\ge 4$ and
$$
\frac{1}{n}\binom{2t+n-1}{2t}> 2 \binom{n+t-1}{t-1}+\binom{n-1+t}{t},
$$

\noindent
then the generic form in  $n$ variables of
degree $s$ defines an Artinian Gorenstein
algebra that is not the Artinian reduction of a set of $\rho_n(s)$ Gorenstein points of $\mathbb P^n_{\res}$.
\end{proposition}
\begin{proof}

\medskip
\noindent
 Let $F $  a generic form of degree $s  $ in $n$ variables, then  by Macaulay's correspondence  it defines a compressed Gorenstein algebra $B$ of length $\rho_n(s).$  By Proposition \ref{reductionPoints} and Remark \ref{How}, for being $B$ the Artinian reduction of a set $X$ of $\rho_n(s)$ distinct points in $\mathbb P^n_{\res}, $ then $F$ must decompose as sum of $\rho_n(s)$ $s$-powers of linear forms.    Under our assumptions we will prove that $G_n(s) > \rho_n(s)  $ and hence the result.
First, we prove this inequality for $s=2 t+1\ge 3$.
By the assumptions  we have
$$
G_n(s)\ge \frac{1}{n}\binom{2t+n}{2t+1}> 2 \binom{n+t}{t}=\rho_n(s).
$$

\noindent
Assume now that $s=2j\ge 3$, then
$$
G_n(s)\ge \frac{1}{n}\binom{2t+n-1}{2t}>
2 \binom{n+t-1}{t-1}+\binom{n-1+t}{t}=\rho_n(s).
$$
\end{proof}

\medskip
We remark  that, according to Proposition \ref{NoLifGor},   given an integer $s$, we may compute an integer $n_0$ such that for all $n \ge n_0$  a generic form of degree $s$ in $n$ indeterminates    defines an Artinian  Gorenstein
algebra which  is not the Artinian  reduction of a set of $\rho_n(s)$ Gorenstein points of $\mathbb P^n_{\res}.$

\vskip 2mm
Here a table with the computation of such integer $n_0.$

\medskip
\begin{center}
\begin{tabular}{|c| c c c c c c c c c c|}
 \hline
 $s=2t+1$ & 3 & 5 & 7 &$\cdots$& 27& 29&$\cdots$&87& 89
 &$\cdots$\\
 \hline
 $n_0$  & 11 & 8 & 7&$\cdots$&7 & 8& $\cdots$&8 & 9 &$\cdots$
\\ \hline \hline
 $s=2t$ & 4 & 6 & 8 &$\cdots$& 24& 26&$\cdots$&86& 88&$\cdots$\\
 \hline
 $n_0$  & 11 & 8 & 7&$\cdots$&7 & 8& $\cdots$&8 & 9 &$\cdots$\\
 \hline
\end{tabular}
\end{center}

\medskip
\noindent
In particular  the assumptions of \propref{NoLifGor} hold true if $n\ge \Max\{t^2+4t+5, 11\}$, with $t=[s/2]$,  but this   bound is in general far to be the optimal.

%%%%%%%%%%%%%%%%%%%%%%%%%%%%%%%%%%%%%%%%%%%%%%%%%%%%%%%%%%%%%%%%%%%%%%%%%%%%%%%%%%%
%%%%%%%%%%%%%%%%%%%%%%%%%%%%%%%%%%%%%%%%%%%%%%%%%%%%%%%%%%%%%%%%%%%%%%%%%%%%%%%%%%%
%%%%%%%%%%%%%%%%%%%%%%%%%%%%%%%%%%%%%%%%%%%%%%%%%%%%%%%%%%%%%%%%%%%%%%%%%%%%%%%%%%%
%%%%%%%%%%%%%%%%%%%%%%%%%%%%%%%%%%%%%%%%%%%%%%%%%%%%%%%%%%%%%%%%%%%%%%%%%%%%%%%%%%%
\bigskip
\bibliographystyle{amsplain}

\end{document}